\documentclass[10pt]{article}
\usepackage{amsmath,amssymb,mathtools,amsthm,textcomp,float,mathscinet
}
\usepackage{amsfonts,graphicx,enumerate,listings,csquotes}
\usepackage{color,tikz}
\usepackage[mathscr]{eucal}
\theoremstyle{definition}
\usepackage[american]{babel}
\newtheorem{theorem}{\bf Theorem}[section]
\newtheorem{remark}{\bf Remark}[section]

\newtheorem{lemma}{Lemma}[section]
\newtheorem{corollary}{Corollary}[section]

\newtheorem{definition}{Definition}[section]
\newtheoremstyle
    {remarkstyle}
    {}
    {11pt}
    {}
    {}
    {\bfseries}
    {:}
    {     }
    {\thmname{#1} \thmnumber{#2} }

\theoremstyle{remarkstyle}

\begin{document}

\title{Space-fractional versions of the negative binomial and Polya-type
processes}
\author{L. Beghin \thanks{
Address: Department of Statistical Sciences, Sapienza University of Rome,
P.le A. Moro 5, I-00185 Roma, Italy. e-mail: \texttt{luisa.beghin@uniroma1.it}} \and P. Vellaisamy \thanks{%
Address: Indian Institute of Technology Bombay, Powai, Mumbai, India-400076.
e-mail: \texttt{pv@math.iitb.ac.in}}}
%\author{ P. Vellaisamy}

\date{}
\maketitle

\begin{abstract}
In this paper, we introduce a space fractional negative binomial (SFNB)
process by subordinating the space fractional Poisson process to a gamma
subordinator. Its one-dimensional distributions are derived in terms of
generalized Wright functions and their governing equations are obtained. It
is a L\'{e}vy process and the corresponding L\'{e}vy measure is given.
Extensions to the case of distributed order SFNB process, where the
fractional index follows a two-point distribution, is analyzed in detail.
The connections of the SFNB process to a space fractional Polya-type process
is also pointed out. Moreover, we define and study a multivariate version of
the SFNB obtained by subordinating a $d$-dimensional space-fractional
Poisson process by a common independent gamma subordinator.
Some applications of the SFNB process to the
studies of population's growth and epidemiology are pointed out.
% More precisely,
%our model can be applied to grouped populations with random sizes
%and groups'  counts rising in time as a non-orderly process.
Finally, we discuss an algorithm for the simulation of the SFNB process.
\end{abstract}

\vspace*{.3cm} \noindent \emph{AMS Subject Classification (2010):} Primary:
60G22, Secondary: 60G51, 60E05

\vspace*{.3cm} \noindent \emph{Keywords:} Fractional negative binomial
process; stable subordinator; Wright function; Polya-type process; governing
equations.

\section{Introduction}

Recently, the fractional versions of the Poisson and the negative binomial
processes have received a lot of attention of the researchers. These
fractional processes have been found more useful, than the classical ones,
for modelling several phenomena that occur in various disciplines. One of
the aims of this paper is to introduce a new fractional version of the
well-known negative binomial (NB) process. The latter is widely applied in
many different fields, mainly for its property of overdispersion; see \cite%
{KOZ}, \cite{UV} and \cite{VAL} (where the more general Tweedie-Poisson
model is studied). A first fractional version of the NB process has been
introduced by Vellaisamy and Maheshwari in \cite{VEL} and we refer to it as
a time-fractional variant, since it is defined by means of the so-called
\textquotedblleft time-fractional Poisson process". Other and different
fractional versions have been defined in \cite{BEG3} and \cite{BEGMA}, in
terms of the fractional gamma process; see \cite{VEL} for the differences
among these approaches.

\vspace*{0.3cm} \noindent Henceforth, $\mathbb{Z_{+}}=\{0,1,\ldots \}$
denotes the set of nonnegative integers. Let
\begin{equation*}
\binom{\alpha }{n}:=\left\{
\begin{array}{l}
\frac{\Gamma (\alpha +1)}{\Gamma (\alpha -n+1)n!},\qquad \alpha -n+1\notin
\mathbb{Z}_{-} \\
0,\qquad \alpha -n+1\in \mathbb{Z}_{-}%
\end{array}%
\right. ,
\end{equation*}
for any $\alpha \in \mathbb{R},$ $n\in \mathbb{Z_{+}}$ and let $X$
be a negative binomial r.v. with parameters $\gamma >0$ and $\eta
\in (0,1)$, with
distribution%
\begin{equation}
\mathbb{P}(X=n)=\binom{n+\gamma -1}{n}\eta ^{n}(1-\eta )^{\gamma },\quad
n\in \mathbb{Z}_{+},  \label{nb}
\end{equation}%
which we denote by $NB(\gamma ,\eta ).$ When $\gamma =k$, a positive
integer, $X$ denotes the number of successes before the $k$-th failure in a
sequence of Bernoulli trials with success probability $\eta $.

\vspace*{0.3cm} \noindent Let $\{N(t,\lambda )\}_{t\geq 0}$ be a homogeneous
Poisson process with rate $\lambda >0$, and $\{Y(t)\}_{t\geq 0}$ be an
independent gamma subordinator, where $Y(t)\sim G(\alpha ,pt)$, the gamma
distribution with scale parameter $\alpha ^{-1}$ and shape parameter $pt>0$
(see eq. (\ref{gdensity}) below). Then the process $\{Q(t,\lambda )\}_{t\geq
0}=\{N(Y(t),\lambda )\}_{t\geq 0}$ is called a negative binomial process and
$Q(t,\lambda )\sim NB(pt,\eta )$, for $t>0$, where $\eta =\lambda /(\alpha
+\lambda )$ (see \cite{KOZ} and \cite{VEL}).

\vspace*{0.3cm} \noindent Let $\left\{ N_{\beta }(t,\lambda )\right\}
_{t\geq 0}$ be the time-fractional Poisson process (see \cite{LAS}, \cite%
{MEE}), where $N_{\beta }(t,\lambda )=N(E_{\beta }(t),\lambda )$ and $%
\{E_{\beta }(t)\}_{t\geq 0}$ is the inverse-stable subordinator. \newline
Recently, the time-fractional NB process $\{Q_{\beta }(t,\lambda )\}_{t\geq
0}=\left\{ N_{\beta }(Y(t),\lambda )\right\} _{t\geq 0}$ was introduced and
its properties investigated detail in \cite{VEL}. Let $\left\{ D_{\beta
}(t)\right\} _{t\geq 0}$, $0<\beta <1$, be a $\beta $-stable subordinator.
Then the space-fractional Poisson process (SFPP) defined as $\left\{
\overline{N}_{\beta }(t,\lambda )\right\} _{t\geq 0}=\left\{ N(D_{\beta
}(t),\lambda )\right\} _{t\geq 0}$ was introduced and discussed in \cite{ORS}%
. Therefore, it is of interest to define the space-fractional negative
binomial (SFNB) process, analogous to SFPP (see Def.4 below). Note that the
SFNB process would be a L\'{e}vy process, unlike the time-fractional NB
process studied in \cite{VEL}.

\noindent In Section 2, we introduce the SFNB process, obtain its
one-dimensional distributions and also derive the underlying governing
equations.
Moreover, we prove that it is equivalent, in the sense
of the finite-dimensional distributions, to a compound-sum
process, where the leading counting process is represented by a
SFPP. This representation shows that the SFNB process could be
to be suitable for some biological applications, in the
analysis of groups of individuals distributed randomly in space
(or time). More precisely, let the number of colonies be
represented by a space-fractional Poisson process and let the
number of individuals in the colonies be distributed independently
(according to a generalization of the logarithmic distribution).
Then the total size of the population follows a SFNB process. The
use of NB distribution in biological data dates back to the
nineteen fifties (see, e.g., \cite{BLI} and \cite{ANS}).
 The SFNB process could be useful for practical situations where the
evolution of the number of groups or colonies is not an orderly
process, { \it i.e.}, it can increase  more than one unit in a
infinitely small interval of time with positive probability. On
the contrary, when modeling this kind of population evolutions by
the standard NB process, multiple occurrences are not allowed
(see, e.g. \cite{KOZ} and the references therein).

\noindent In Section 3, we present two different generalizations of the
previous results: the first one is obtained by taking the fractional index $%
\beta $ random instead of constant in $(0,1)$. In particular, we will
consider the case of a two-valued discrete random variable. The second
extension is obtained by considering the following subordination $\overline{W%
}_{\beta }(t,u)=\overline{N}_{\beta }(u,Y^{d}(t))$, where $\left\{ \overline{%
N}_{\beta }(\cdot ,Y^{d}(t))\right\} _{t\geq 0}$ is the space-fractional
Poisson process with $\lambda $ replaced by the process $Y^{d}(t)$ and $\ d$
is a positive constant. We call $\left\{ \overline{W}_{\beta }(t))\right\}
_{t\geq 0}$ space-fractional \textquotedblleft Polya-type" process, in
analogy to \cite{VEL}.

\noindent Finally, in Section 4, we consider the multidimensional extension
of the SFNB process, by subordinating the space-fractional Poisson process
studied in \cite{BEGM} by a common gamma subordinator. The process is a
multidimensional L\'{e}vy process, for which we evaluate the corresponding L%
\'{e}vy measure. For all the above mentioned processes, we give the
one-dimensional distributions and the corresponding governing equations.
\newline

\section{Space fractional negative binomial process}

First, we introduce some preliminary definitions and results that will be
used later.

\subsection{Preliminary definitions and results}

We start with the space-fractional Poisson process (SFPP) $\left\{ \overline{%
N}_{\beta }(t,\lambda )\right\} _{t\geq 0}=\left\{ N(D_{\beta }(t),\lambda
)\right\} _{t\geq 0}$ . Its one-dimensional distributions are given by (see
\cite{ORS})
\begin{equation}
\overline{p}_{_{\beta }}(n|t,\lambda )=\mathbb{P}[\overline{N}_{\beta
}(t,\lambda )=n]=\frac{(-1)^{n}}{n!}\sum\limits_{k=0}^{\infty }\frac{%
(-\lambda ^{\beta }t)^{k}}{k!}\frac{\Gamma (\beta k+1)}{\Gamma (\beta k+1-n)}%
,~t\geq 0.  \label{neqn2}
\end{equation}%
It solves the fractional difference-differential equation (\cite[eq. (2.4)]%
{ORS}) defined by
\begin{align}
\frac{\partial }{\partial t}\overline{p}_{_{\beta }}(n|t,\lambda )&
=-\lambda ^{\beta }(1-B)^{\beta }\overline{p}_{_{\beta }}(n|t,\lambda
),~~~~\beta \in (0,1],  \label{neqn3} \\
\overline{p}_{_{\beta }}(n|0,\lambda )& =\left\{
\begin{array}{ll}
1, & \text{for }n=0,\nonumber \\
0, & \mbox{for }n>0,%
\end{array}%
\right.  \notag
\end{align}%
where $B$ is the backward shift operator defined by $Bu(n)=u(n-1)$ for any
function $u:\mathbb{N}\rightarrow \mathbb{R}.$

We will derive the explicit expression of the distribution of the SFNB
process in terms of the generalized Wright function which is defined as (see
\cite{KIL})
\begin{equation*}
_{p}\Psi _{q}\left[ \left. z\right\vert _{(b_{j},\beta
_{j})_{1,q}}^{(a_{i},\alpha _{i})_{1,p}}\right] :=\sum_{k=0}^{\infty }\frac{%
\prod\limits_{l=1}^{p}\Gamma (a_{i}+\alpha _{i}k)}{\prod\limits_{j=1}^{q}%
\Gamma (b_{j}+\beta _{j}k)}\frac{z^{k}}{k!},
\end{equation*}%
for $z,a_{i},b_{j}\in \mathbb{C}$, $\alpha _{i},\beta _{j}\in \mathbb{R}$, $%
i=1,...,p,$ $j=1,...,q.$ In view of what follows we introduce the following
symbols: $\Delta :=\sum_{j=1}^{q}\beta _{j}-\sum_{j=1}^{p}\alpha _{j}$ and $%
\delta :=\prod\limits_{l=1}^{p}|\alpha _{l}|^{-\alpha
_{l}}\prod\limits_{j=1}^{q}|\beta _{j}|^{\beta _{j}}.$

\noindent Also, we will obtain the space-fractional differential equation
satisfied by the distribution of the SFNB process in terms of the shift
operator
\begin{equation}
e^{cD_{x}}f(x):=\sum_{n=0}^{\infty }\frac{c^{n}D_{x}^{n}}{n!}f(x)=f(x+c),
\label{shi}
\end{equation}%
where $D_{x}=d/dx$, which is defined for any analytic function $f:\mathbb{R}%
\rightarrow \mathbb{R}$ and $c\in \mathbb{R}.$\newline

\noindent Let $\{Y(t)\}_{t \geq 0}$ denote the gamma subordinator, where $%
Y(t) \sim G(\alpha ,pt)$, the gamma distribution, with density
\begin{equation}  \label{gdensity}
g(z|\alpha ,pt)=\frac{\alpha ^{pt}}{\Gamma (pt)}z^{pt-1}e^{-\alpha z},\qquad
\alpha ,p,z,t>0.
\end{equation}
We need also the following result on the differential equation satisfied by
the density of the gamma subordinator, proved in \cite{BEG2}.

\begin{lemma}
\label{ecp}The one-dimensional density $g(x|\alpha ,pt)$, defined in %
\eqref{gdensity}, of the gamma subordinator $\left\{ Y(t)\right\} _{t\geq 0}$
satisfies the following Cauchy problem, for $x,t\geq 0,$%
\begin{equation}
\left\{
\begin{array}{l}
\frac{\partial }{\partial x}g(x|\alpha ,pt)=-\alpha (1-e^{-\frac{1}{p}%
\partial _{t}})g(x|\alpha ,pt) \\
g(x|\alpha ,0)=\delta (x) \\
\lim_{|x|\rightarrow +\infty }g(x|\alpha ,pt)=0,%
\end{array}%
\right.  \label{res}
\end{equation}%
where $e^{-\frac{1}{p}\partial _{t}}$ is the partial derivative version of
the shift operator defined in (\ref{shi}), for $c=-1/p,$ and $\delta (x)$ is
the Dirac delta function.
\end{lemma}

An alternative differential equation satisfied by the one-dimensional
distributions of the SFNB process is given in terms of fractional
derivatives and thus we recall the following.

\begin{definition}[Riemann-Liouville fractional derivative]
Let $m\in \mathbf{Z}_{+}\backslash\{0\}$ and $\nu\geq0$. If $f(t)\in
AC^{m-1}[0,T],$ $0\leq t\leq T$, then the (left-hand) Riemann-Liouville
(R-L) fractional derivative $\partial_{t}^{\nu}(\cdot)$ is defined by
\begin{equation}  \label{rld}
\partial_{t}^{\nu}f(t):=
\begin{cases}
\hfill \dfrac{1}{\Gamma(m-\nu)}\dfrac{d^{m}}{dt^{m}}\displaystyle%
\int\limits_{0}^{t}\dfrac{f(s)}{(t-s)^{\nu+1-m}}ds, \hfill & m-1<\nu<m , \\
&  \\
\dfrac{d^{m}}{dt^{m}}f(t), \,\,\,\,\,\,\,\,\,\,\, \nu=m, &
\end{cases}%
\end{equation}
where $AC^{n}[0,T]$ denotes the space of absolutely continuous functions
whose $(n-1)$-th derivatives are also continuous on $[0,T]$.
\end{definition}

We need also the following result which is Lemma 4.2 of \cite{VEL}.

\begin{lemma}
\label{lem4} \label{rlfdegamma} The one-dimensional density $g(x|\alpha ,pt)$%
, given in \eqref{gdensity}, of the gamma subordinator $\left\{ Y(t)\right\}
_{t\geq 0}$ satisfies also the following time-fractional differential
equation, for any $\nu\geq0$:
\begin{align}
\partial _{t}^{\nu }g(y|\alpha ,pt)& =p\partial _{t}^{\nu -1}[\log (\alpha
y)-\psi (pt)]g(y|\alpha ,pt),\,\,\,y>0,  \label{rlfdeeq} \\
g(y|\alpha ,0)& =0 ,  \label{gg}
\end{align}%
where $\psi (x):=\Gamma ^{\prime }(x)/\Gamma (x)$ is the digamma function
and $\partial _{t}^{\nu }(\cdot )$ is the R-L derivative defined in %
\eqref{rld}.
\end{lemma}

\subsection{The SFNB process and its properties}

We are now ready to introduce the SFNB process in terms of the SFPP as
follows:

\begin{definition}
Let $\left\{ \overline{N}_{\beta }(t,\lambda )\right\} _{t\geq 0}$ be the
SFPP. Then the process defined by
\begin{equation}
\overline{Q}_{\beta }(t,\lambda )=\overline{N}_{\beta }(Y(t),\lambda ),
\label{one}
\end{equation}%
where $\left\{ Y(t)\right\} _{t\geq 0}$ is an independent gamma
subordinator, is called the SFNB process.
\end{definition}

The one-dimensional distribution of (\ref{one}) can be written as
\begin{equation}
\overline{\delta }_{\beta }(n|\alpha ,pt,\lambda ):=\mathbb{P}\left\{
\overline{Q}_{\beta }(t,\lambda )=n\right\} =\int_{0}^{+\infty }\overline{p}%
_{\beta }(n|z,\lambda )g(z|\alpha ,pt)dz,  \label{two}
\end{equation}%
where $\overline{p}_{\beta }(n|z,\lambda )$ is the probability mass function
(pmf) of SFPP defined in \eqref{neqn2} and $g(z|\alpha, pt)$ is the density
of $\Gamma (\alpha, pt)$, defined in \eqref{gdensity}.

We next present our first result on the SFNB process.

\begin{theorem}
The one-dimensional distribution of the process $\left\{ \overline{Q}_{\beta
}(t,\lambda )\right\} _{t\geq 0}$ is given by%
\begin{equation}
\overline{\delta }_{\beta }(n|\alpha ,pt,\lambda )=\frac{(-1)^{n}}{n!}\frac{1%
}{\Gamma (pt)}\,_{2}\Psi _{1}\left[ \left. -\frac{\lambda ^{\beta }}{\alpha }%
\right\vert _{(1-n,\beta )}^{(1,\beta )\;(pt,1)}\right] ,\quad n\in \mathbb{Z%
}^{+},\text{ }t\geq 0  \label{dis}
\end{equation}%
and satisfies the following space-fractional equation%
\begin{equation}
(e^{-\frac{1}{p}\partial _{t}}-1)\overline{\delta }_{\beta }(n|\alpha
,pt,\lambda )=\frac{\lambda ^{\beta }}{\alpha }(I-B)^{\beta }\overline{%
\delta }_{\beta }(n|\alpha ,pt,\lambda ),  \label{eq}
\end{equation}%
with initial condition $\overline{\delta }_{\beta }(n|\alpha ,0,\lambda
)=1_{\left\{ n=0\right\} }$ and $B$ is the backward operator.
\end{theorem}

\begin{proof}
Formula (\ref{dis}) can be derived by applying (\ref{two}) together with
equation (\ref{neqn2}):
\begin{eqnarray}
\overline{\delta }_{\beta }(n|\alpha ,pt,\lambda ) &=&\frac{(-1)^{n}}{n!}%
\frac{\alpha ^{pt}}{\Gamma (pt)}\sum_{r=0}^{\infty }\frac{(-\lambda ^{\beta
})^{r}}{r!}\frac{\Gamma (\beta r+1)}{\Gamma (\beta r+1-n)}\int_{0}^{+\infty
}z^{r+pt-1}e^{-\alpha z}dz  \notag \\
&=&\frac{(-1)^{n}}{n!}\frac{1}{\Gamma (pt)}\sum_{r=0}^{\infty }\frac{%
(-\lambda ^{\beta }/\alpha )^{r}}{r!}\frac{\Gamma (\beta r+1)\Gamma (r+pt)}{%
\Gamma (\beta r+1-n)}.  \label{cc}
\end{eqnarray}%
The series converges for $|\lambda ^{\beta }/\alpha |<1,$ by Theorem 1.5,
p.56 in \cite{KIL} with $\Delta =\beta -(\beta +1)=-1,$\ $\delta =|\beta
|^{-\beta }|\beta |^{\beta }=1.$ \newline

In order to prove equation (\ref{eq}), we apply formula (\ref{two}) together
with Lemma 1, as follows:
\begin{eqnarray*}
e^{-\frac{1}{p}\partial _{t}}\overline{\delta }_{\beta }(n|\alpha
,pt,\lambda ) &=&\int_{0}^{+\infty }\overline{p}_{\beta }(n|z,\lambda )e^{-%
\frac{1}{p}\partial _{t}}g(z|\alpha ,pt)dz \\
&=&\overline{\delta }_{\beta }(n|\alpha ,pt,\lambda )+\frac{1}{\alpha }%
\int_{0}^{+\infty }\overline{p}_{\beta }(n|z,\lambda )\frac{\partial }{%
\partial z}g(z|\alpha ,pt)dz \\
&=&\overline{\delta }_{\beta }(n|\alpha ,pt,\lambda )-\frac{1}{\alpha }%
\int_{0}^{+\infty }\frac{\partial }{\partial z}\overline{p}_{\beta
}(n|z,\lambda )g(z|\alpha ,pt)dz \\
&=&\overline{\delta }_{\beta }(n|\alpha ,pt,\lambda )+\frac{\lambda ^{\beta }%
}{\alpha }(I-B)^{\beta }\int_{0}^{+\infty }\overline{p}_{\beta }(n|z,\lambda
)g(z|\alpha ,pt)dz,
\end{eqnarray*}%
using (\ref{neqn3}). Note that integration by parts is used to obtain the
third step above. The initial condition is satisfied by (\ref{dis}), as can
be checked by considering that, for $t=0,$ the probability (\ref{cc})
vanishes for any $n>0$, while for $n=0$ it can be written as follows%
\begin{equation*}
\overline{\delta }_{\beta }(0|\alpha ,pt,\lambda )=\sum_{r=0}^{\infty
}(-\lambda ^{\beta }/\alpha )^{r}\binom{r+pt-1}{r}=\frac{1}{(1+\lambda
^{\beta }/\alpha )^{pt}}\text{, }
\end{equation*}%
which is equal to one, for $t=0.$ In the last step, we have applied the
following well-known combinatorial identity (see e.g. \cite{SRI}):%
\begin{equation}
\sum_{r=0}^{\infty }a^{r}\binom{r+\xi -1}{r}=(1-a)^{-\xi }.  \label{id}
\end{equation}
\end{proof}

\begin{remark}
(i) Note that for $\beta =1$, equation (\ref{eq}) reduces to formula (69) in
\cite{BEG3}.

(ii) Formula (\ref{dis}) represents a proper distribution, which can be seen
as follows:
\begin{eqnarray*}
\sum_{n=0}^{\infty }\overline{\delta }_{\beta }(n|\alpha ,pt,\lambda ) &=&%
\frac{1}{\Gamma (pt)}\sum_{j=0}^{\infty }\frac{(-\lambda ^{\beta })^{j}}{%
j!\alpha ^{j}}\Gamma (j+pt)\sum_{n=0}^{\infty }\frac{(-1)^{n}}{n!}\frac{%
\Gamma (\beta j+1)}{\Gamma (\beta j+1-n)} \\
&=&\frac{1}{\Gamma (pt)}\sum_{j=0}^{\infty }\frac{(-\lambda ^{\beta })^{j}}{%
j!\alpha ^{j}}\Gamma (j+pt)(1-1)^{\beta j}=1
\end{eqnarray*}

Also, when $\beta =1$, we can check that (\ref{dis}) reduces to the
distribution of the negative binomial distribution. Indeed, we can write%
\begin{eqnarray*}
\overline{\delta }_{1}(n|\alpha ,pt,\lambda ) &=&\frac{(-1)^{n}}{n!}\frac{1}{%
\Gamma (pt)}\sum_{j=n}^{\infty }\left( -\frac{\lambda }{\alpha }\right) ^{j}%
\frac{\Gamma (j+pt)}{(j-n)!} \\
&=&\frac{\Gamma (n+pt)}{n!\Gamma (pt)}\left( \frac{\lambda }{\alpha }\right)
^{n}\sum_{l=0}^{\infty }\left( -\frac{\lambda }{\alpha }\right) ^{l}\binom{%
n+l+pt-1}{l}~~(\text{using (\ref{id})}) \\
&=&\left( \frac{\lambda }{\alpha +\lambda }\right) ^{n}\left( \frac{\alpha }{%
\alpha +\lambda }\right) ^{pt}\binom{n+pt-1}{n},
\end{eqnarray*}%
which is the distribution of $Q(t,\lambda )\sim NB(pt,\lambda /(\alpha
+\lambda )).$
\end{remark}

\noindent An alternative fractional pde satisfied by the distribution (\ref%
{dis}) can be obtained by applying Lemma 3.

\begin{theorem}
The distribution of the SFNB process $\bar{\delta}_{\beta }\left( n|\alpha
,pt,\lambda \right) $ satisfies the following fractional pde:
\begin{equation*}
\frac{1}{p}\partial _{t}^{\gamma }\bar{\delta}_{\beta }\left( n|\alpha
,pt,\lambda \right) =\partial _{t}^{\gamma -1}[\log \alpha -\psi (pt)]\bar{%
\delta}_{\beta }\left( n|\alpha ,pt,\lambda \right) +\int_{0}^{\infty }(\log
y)\bar{p}_{\beta }\left( n|y,\lambda \right) \partial _{t}^{\gamma
-1}g\left( y|\alpha ,pt\right) dy
\end{equation*}%
with $\gamma \geq 0,$ $\psi (x):=\Gamma ^{\prime }(x)/\Gamma (x)$\ and $\bar{%
\delta}_{\beta }\left( n|\alpha ,0,\lambda \right) =1_{\left\{ n=0\right\} }$%
.
\end{theorem}

\begin{proof}
By considering (\ref{rlfdeeq})-(\ref{gg}) we get
\begin{eqnarray*}
\partial _{t}^{\gamma }\bar{\delta}_{\beta }\left( n|\alpha ,pt,\lambda
\right) &=&\partial _{t}^{\gamma }\int_{0}^{\infty }\bar{p}_{\beta }\left(
n|y,\lambda \right) g\left( y|\alpha ,pt\right) dy \\
&=&\int_{0}^{\infty }\bar{p}_{\beta }\left( n|y,\lambda \right) \partial
_{t}^{\gamma }g\left( y|\alpha ,pt\right) dy \\
&=&p\int_{0}^{\infty }\bar{p}_{\beta }\left( n|y,\lambda \right) \partial
_{t}^{\gamma -1}\left[ \log \alpha +\log y-\psi (pt)\right] g\left( y|\alpha
,pt\right) dy \\
&=&p\partial _{t}^{\gamma -1}\log \alpha \bar{\delta}_{\beta }\left(
n|\alpha ,pt,\lambda \right) +p\partial _{t}^{\gamma -1}\int_{0}^{\infty
}(\log y)\bar{p}_{\beta }\left( n|y,\lambda \right) g\left( y|\alpha
,pt\right) dy \\
&&-p\partial _{t}^{\gamma -1}\psi (pt)\bar{\delta}_{\beta }\left( n|\alpha
,pt,\lambda \right) ,
\end{eqnarray*}%
which proves the result.
\end{proof}

The SFNB process (\ref{one}) is defined by time-changing the process $%
\left\{ \overline{N}_{\beta }(t,\lambda )\right\} _{t\geq 0}$ by means of
the gamma subordinator. But, $\left\{ \overline{N}_{\beta }(t,\lambda
)\right\} _{t\geq 0}$ is itself a subordinator and the Laplace transform of
its probability mass function can be derived from (2.12) of \cite{ORS}, i.e.%
\begin{equation*}
\sum_{n=0}^{\infty }e^{-un}\overline{p}_{\beta }(n|z,\lambda )=e^{-\lambda
^{\beta }z(1-e^{-u})^{\beta }}.
\end{equation*}%
As a consequence, $\left\{ \overline{Q}_{\beta }(t,\lambda )\right\} _{t\geq
0}$ is a subordinator with Laplace transform of the distribution equal to
\begin{eqnarray*}
\mathbb{E}e^{-u\overline{Q}_{\beta }(t,\lambda )} &=&\sum_{n=0}^{\infty
}e^{-un}\overline{\delta }_{\beta }(n|\alpha ,pt,\lambda )=\int_{0}^{+\infty
}e^{-\lambda ^{\beta }z(1-e^{-u})^{\beta }}g(z|\alpha ,pt)dz \\
&=&\exp \left\{ -pt\ln \left( 1+\frac{\lambda ^{\beta }(1-e^{-u})^{\beta }}{%
\alpha }\right) \right\} .
\end{eqnarray*}%
Thus the Laplace exponent of $\left\{ \overline{Q}_{\beta }(t,\lambda
)\right\} _{t\geq 0}$ is given by
\begin{equation*}
\psi (u):=-\frac{1}{t}\ln \left( \mathbb{E}e^{-u\overline{Q}_{\beta
}(t,\lambda )}\right) =p\ln \left( 1+\frac{\lambda ^{\beta
}(1-e^{-u})^{\beta }}{\alpha }\right)
\end{equation*}%
and the discrete L\'{e}vy measure can be derived as follows.

\begin{theorem}
The discrete L\'{e}vy measure of the SFNB process is given by%
\begin{equation}
\nu _{\beta }(\cdot )=p\sum_{k=1}^{\infty }(-1)^{k}\delta _{\left\{
k\right\} }(\cdot )\sum_{r=1}^{\infty }\frac{(-\lambda ^{\beta }/\alpha )^{r}%
}{r}\binom{\beta r}{k},  \label{le}
\end{equation}%
where $\delta _{\left\{ k\right\} }(\cdot )$ is the Dirac measure
concentrated at $k.$
\end{theorem}

\begin{proof}
Since the SFNB process is a subordinated process (see (\ref{one})), we can
apply Theorem 30.1, p. 197 in \cite{SAT}, which gives a calculation rule for
the L\'{e}vy measure in the case of subordinated L\'{e}vy processes. Let $%
\mu _{\Gamma }$ be the L\'{e}vy measure of the gamma subordinator $\left\{ Y
(t)\right\} _{t\geq 0}$ (i.e. $\mu _{\Gamma }(s)=ps^{-1}e^{-\alpha s}$).
Then we get%
\begin{eqnarray*}
\nu _{\beta }(\cdot ) &=&\int_{0}^{+\infty }\sum_{k=1}^{\infty }\overline{p}%
_{\beta }(k|s,\lambda ) \delta _{\left\{ k\right\} }(\cdot )\mu _{\Gamma
}(s)ds~~ (\text{using (2)}) \\
&=&p\sum\limits_{k=1}^{\infty }\frac{(-1)^{k}}{k!}\delta _{\left\{ k\right\}
}(\cdot )\sum_{r=0}^{\infty }\frac{(-\lambda ^{\beta })^{r}}{r!}\frac{\Gamma
(\beta r+1)}{\Gamma (\beta r+1-k)}\int_{0}^{+\infty }s^{r-1}e^{-\alpha s}ds
\\
&=&p\sum\limits_{k=1}^{\infty }\frac{(-1)^{k}}{k!}\delta _{\left\{ k\right\}
}(\cdot )\sum_{r=1}^{\infty }\frac{(-\lambda ^{\beta }/\alpha )^{r}}{r}\frac{%
\Gamma (\beta r+1)}{\Gamma (\beta r+1-k)}.
\end{eqnarray*}%
Note that in the last step we have used the fact that the inner sum is zero
for $r=0,$ since $\Gamma (1-k)=\infty $ for any $k\geq 1.$
\end{proof}

It is easy to check that, in the special case $\beta =1,$ formula (\ref{le})
coincides with the well-known L\'{e}vy measure of the NB process:%
\begin{eqnarray*}
\nu _{1}(\cdot ) &=&p\sum\limits_{k=1}^{\infty }(-1)^{k}\delta _{\left\{
k\right\} }(\cdot )\sum_{r=k}^{\infty }\frac{(-\lambda /\alpha )^{r}}{r}%
\binom{r}{k} \\
&=&p\sum\limits_{k=1}^{\infty }(-1)^{k}\delta _{\left\{ k\right\} }(\cdot
)\sum_{l=0}^{\infty }\frac{(-\lambda /\alpha )^{l+k}}{l+k}\binom{l+k}{k} \\
&=&p\sum\limits_{k=1}^{\infty }\frac{(\lambda /\alpha )^{k}}{k}\delta
_{\left\{ k\right\} }(\cdot )\sum_{l=0}^{\infty }(-\lambda /\alpha )^{l}%
\binom{l+k-1}{l} \\
&=&p\sum\limits_{k=1}^{\infty }\frac{\delta _{\left\{ k\right\} }(\cdot )}{k}%
\left( \frac{\lambda }{\alpha +\lambda }\right) ^{k}.
\end{eqnarray*}%
We next present  the following alternative construction, as a compound-sum process,   of the
SFNB process $\left\{ \overline{Q}_{\beta }(t,\lambda )\right\} _{t\geq 0}$, as

\begin{equation}
\overline{Q}_{\beta }(t,\lambda )=\sum_{j=1}^{\overline{N}_{\beta
}(t,\lambda )}X_{j},  \label{new}
\end{equation}%
where the $X_{j},$ for $j=1,2,...$, are i.i.d. r.v.'s with probability generating
function (p.g.f.)
\begin{equation}
G_{X}(u|\alpha ,\beta ):=\mathbb{E}u^{X}=1-\frac{1}{\lambda ^{\beta }}\left[
\log \left( 1+\frac{\lambda ^{\beta }}{\alpha }(1-u)^{\beta }\right) \right]
^{1/\beta },\text{ }|u|\leq 1,\beta \in (0,1].  \label{ipp}
\end{equation}%
It can be checked that the p.g.f. of (\ref{new}) coincides with that of (\ref{one}%
), by recalling the Wald formula, for the compound sum $\sum_{j=1}^{N}X_{j}$:%
\begin{equation}
G_{\sum_{j=1}^{N}X_{j}}(u)=G_{N}\left( G_{X}(u)\right) ,  \label{wal}
\end{equation}%
if $X_{j},$ for $j=1,2,...$, are r.v.'s independent from each other and from $%
N$ and with p.g.f. $G_{X}(u)$. Indeed, from (\ref{cc}), putting for
simplicity $p=1$, we have that%
\begin{align}
G_{\overline{Q}_{\beta }(t)}(u|\alpha ,\beta ) &:=\mathbb{E}u^{\overline{Q}%
_{\beta }(t,\lambda )}=\sum_{n=0}^{\infty }(-u)^{n}\sum_{r=0}^{\infty
}(-\lambda ^{\beta }/\alpha )^{r}\binom{\beta r}{n}\binom{r+t-1}{r}
\label{coi} \\
&=\sum_{r=0}^{\infty }\left( -\frac{\lambda ^{\beta }(1-u)^{\beta }}{\alpha
}\right) ^{r}\binom{r+t-1}{r}  \notag \\
&=\frac{1}{\left( 1+\lambda ^{\beta }(1-u)^{\beta }/\alpha \right) ^{t}}%
=\exp \left[ -t\log \left( 1+\frac{\lambda ^{\beta }}{\alpha }(1-u)^{\beta
}\right) \right] .  \notag
\end{align}%
On the other hand, by using the definition (\ref{new}), together with (\ref%
{wal}), we get that%
\begin{eqnarray*}
G_{\overline{Q}_{\beta }(t)}(u|\alpha ,\beta ) &=&G_{\overline{N}_{\beta
}}\left( G_{X}(u|\alpha ,\beta )\right)  \\
&=&[\text{by (2.12) in \cite{ORS}}] \\
&=&\exp \left[ -\lambda ^{\beta }t\left( 1-G_{X}(u|\alpha ,\beta )\right)
^{\beta }\right] ,
\end{eqnarray*}%
which coincides with (\ref{coi}), by taking into account (\ref{ipp}).

As a consequence of the compound-sum representation (\ref{new}), the SFNB
process can represent a useful tool in biological studies of grouped
populations, when the evolution of the number of groups is not an orderly
process, {\it i.e.}, it can increase of more than one unit in a infinitely small
interval of time, with positive probability. Indeed the counting process $%
\left\{ \overline{N}_{\beta }(t,\lambda )\right\} _{t\geq 0}$ has this
property, since it can be checked from (\ref{neqn2}) that, for $m=1,2, \ldots,$
\begin{equation*}
P\left( \overline{N}_{\beta }(\Delta t,\lambda )=m\right) \simeq (-1)^{m+1}%
\binom{\beta }{m}\lambda ^{\beta }\Delta t+o(\Delta t),\qquad \Delta
t\rightarrow 0,\beta \in (0,1],
\end{equation*}

\noindent For certain chosen values of $\beta,p,t$ and $\lambda$,
two plots (Figures 1 and 2) of the distribution of the SFNB
process, computed using Mathematica 10 to the order of $10^{-12}$,
are given below.
\begin{figure}[hb]
    \caption{Plot of the {\it pmf} of the SFNB process for $\beta=0.9,\alpha=2,p=2,t=5$ and $\lambda=1.$}
    \centering
    \includegraphics[width=0.9\textwidth]{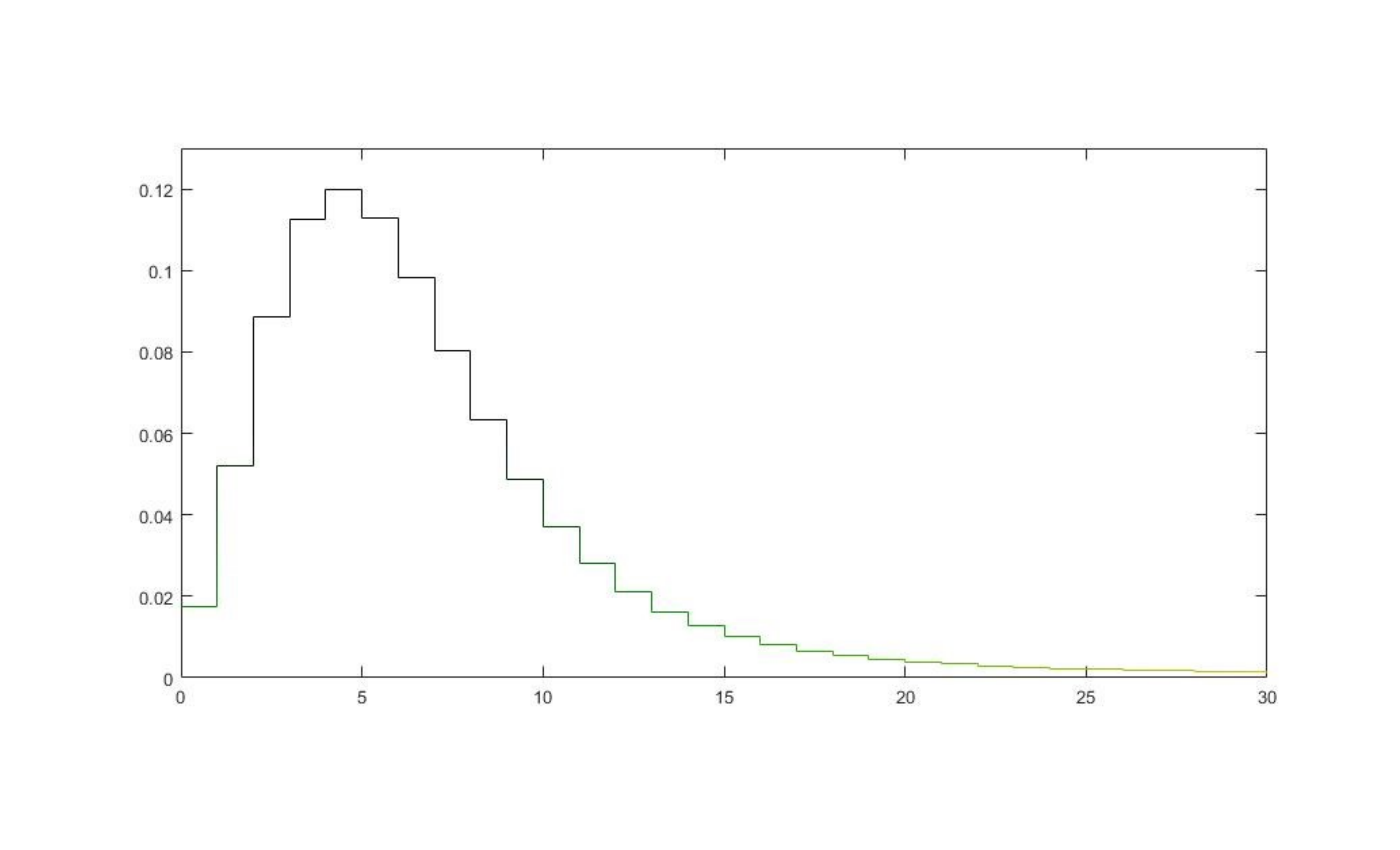}
\end{figure}
\begin{figure}[ht]
    \caption{Plot of the {\it pmf} of the SFNB process for $\beta=0.6,\alpha=3,p=2,t=10$ and $\lambda=2.$}
    \centering
    \includegraphics[width=0.9\textwidth,height=0.35\textheight]{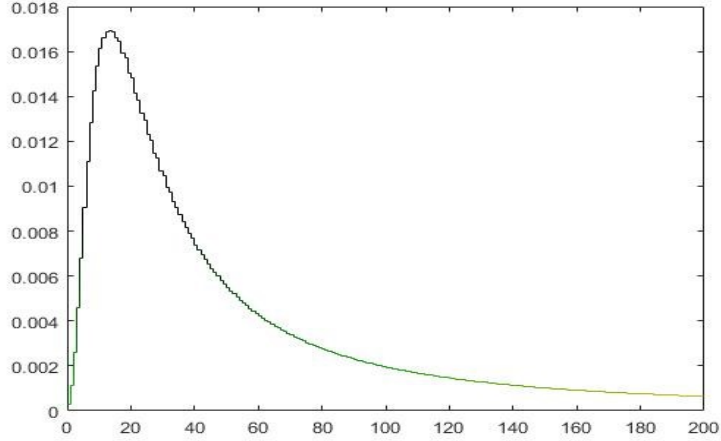}
\end{figure}
%\clearpage}

\section{Generalizations of the SFNB process}

\subsection{SFNB process of distributed order}

A generalization of the results presented in Section 1 can be obtained by
considering the case where the fractional index $\beta $ in the equation (%
\ref{eq}) satisfied by the distribution of the process (\ref{one}) is itself
random. Distributed-order fractional equations has been already treated in
the literature, in the case of superdiffusions and relaxation equations;
see, for example, \cite{CHE}, \cite{MAI1}, \cite{MAI2}, \cite{BEG1}. \newline
Here, we will assume the random index $\beta $ follows two-valued discrete
distribution, namely,
\begin{equation}
f(\beta )=a_{1}\delta (\beta -\beta _{1})+a_{2}\delta (\beta -\beta
_{2}),\qquad \beta _{1},\beta _{2}\in (0,1),  \label{assu}
\end{equation}%
where $\delta (\cdot )$ is the Dirac delta function and $a_{1},a_{2}\geq 0$
with $a_{1}+a_{2}=1.$ Thus, and also in view of (\ref{eq}), we are
interested in the solution of the following equation%
\begin{equation}
(e^{-\frac{1}{p}\partial _{t}}-1)\overline{\delta }_{\beta _{1},\beta
_{2}}(n|\alpha ,pt,\lambda )f(\beta )d\beta =\left( \int_{0}^{1}\frac{%
\lambda ^{\beta }}{\alpha }(I-B)^{\beta }f(\beta )d\beta \right) \overline{%
\delta }_{\beta _{1},\beta _{2}}(n|\alpha ,pt,\lambda )  \label{dist}
\end{equation}%
(with initial condition $\overline{\delta }_{\beta }(n|\alpha ,0,\lambda
)=1_{\left\{ n=0\right\} })$, which, under the assumption (\ref{assu}), can
be rewritten as follows%
\begin{equation}
(e^{-\frac{1}{p}\partial _{t}}-1)\overline{\delta }_{\beta _{1},\beta
_{2}}(n|\alpha ,pt,\lambda )=\left[ \frac{a_{1}\lambda ^{\beta _{1}}}{\alpha
}(I-B)^{\beta _{1}}+\frac{a_{2}\lambda ^{\beta _{2}}}{\alpha }(I-B)^{\beta
_{2}}\right] \overline{\delta }_{\beta _{1},\beta _{2}}(n|\alpha ,pt,\lambda
).  \label{dist2}
\end{equation}

\begin{theorem}
The solution to equation (\ref{dist2}) is given by%
\begin{equation}
\overline{\delta }_{\beta _{1},\beta _{2}}(n|\alpha ,pt,\lambda )=\frac{%
(-1)^{n}}{n!}\frac{1}{\Gamma (pt)}\,\sum_{r=0}^{\infty }\frac{(-a_{1}\lambda
^{\beta _{1}}/\alpha )^{r}}{r!}\,_{2}\Psi _{1}\left[ \left. -\frac{%
a_{2}\lambda ^{\beta _{2}}}{\alpha }\right\vert _{(1-n+\beta _{1}r,\beta
_{2})}^{(1+\beta _{1}r,\beta _{2})\;(pt+r,1)}\right] ,  \label{dist5}
\end{equation}%
for $n\in \mathbb{Z}^{+},$ $t\geq 0.$
\end{theorem}

\begin{proof}
We start by deriving from (\ref{dist2}) the equation satisfied by the
probability generating function $G_{\beta _{1},\beta _{2}}(u|\alpha
,pt,\lambda ):=\sum_{n=0}^{\infty }u^{n}\overline{\delta }_{\beta _{1},\beta
_{2}}(n|\alpha ,pt,\lambda )$:%
\begin{align}
(e^{-\frac{1}{p}\partial _{t}}-1)&G_{\beta _{1},\beta _{2}}(u|\alpha
,pt,\lambda )  \label{dist3} \\
=&\sum_{n=0}^{\infty }u^{n}\left[ \frac{a_{1}\lambda ^{\beta _{1}}}{\alpha }%
(I-B)^{\beta _{1}}+\frac{a_{2}\lambda ^{\beta _{2}}}{\alpha }(I-B)^{\beta
_{2}}\right] \overline{\delta }_{\beta _{1},\beta _{2}}(n|\alpha ,pt,\lambda
)  \notag \\
=&\frac{a_{1}\lambda ^{\beta _{1}}}{\alpha }\sum_{n=0}^{\infty
}u^{n}\sum_{j=0}^{n}(-1)^{j}\binom{\beta _{1}}{j}\overline{\delta }_{\beta
_{1},\beta _{2}}(n-j|\alpha ,pt,\lambda )+  \notag \\
&+\frac{a_{2}\lambda ^{\beta _{2}}}{\alpha }\sum_{n=0}^{\infty
}u^{n}\sum_{j=0}^{n}(-1)^{j}\binom{\beta _{2}}{j}\overline{\delta }_{\beta
_{1},\beta _{2}}(n-j|\alpha ,pt,\lambda )  \notag \\
=&\frac{a_{1}\lambda ^{\beta _{1}}}{\alpha }\sum_{j=0}^{\infty }(-u)^{j}%
\binom{\beta _{1}}{j}\sum_{n=j}^{\infty }u^{n-j}\overline{\delta }_{\beta
_{1},\beta _{2}}(n-j|\alpha ,pt,\lambda )+  \notag \\
&+\frac{a_{2}\lambda ^{\beta _{2}}}{\alpha }\sum_{j=0}^{\infty }(-u)^{j}%
\binom{\beta _{2}}{j}\sum_{n=j}^{\infty }u^{n-j}\overline{\delta }_{\beta
_{1},\beta _{2}}(n-j|\alpha ,pt,\lambda )  \notag \\
=&\left[ \frac{a_{1}\lambda ^{\beta _{1}}}{\alpha }\sum_{j=0}^{\infty
}(-u)^{j}\binom{\beta _{1}}{j}+\frac{a_{2}\lambda ^{\beta _{2}}}{\alpha }%
\sum_{j=0}^{\infty }(-u)^{j}\binom{\beta _{2}}{j}\right] G_{\beta _{1},\beta
_{2}}(u|\alpha ,pt,\lambda )  \notag \\
=&\left[ \frac{a_{1}\lambda ^{\beta _{1}}}{\alpha }(1-u)^{\beta _{1}}+\frac{%
a_{2}\lambda ^{\beta _{2}}}{\alpha }(1-u)^{\beta _{2}}\right] G_{\beta
_{1},\beta _{2}}(u|\alpha ,pt,\lambda ),  \notag
\end{align}%
with the initial condition $G_{\beta _{1},\beta _{2}}(u|\alpha ,0,\lambda
)=1.$ Equation (\ref{dist3}) is satisfied by the following function%
\begin{equation}
G_{\beta _{1},\beta _{2}}(u|\alpha ,pt,\lambda )=\left[ 1+\frac{a_{1}\lambda
^{\beta _{1}}}{\alpha }(1-u)^{\beta _{1}}+\frac{a_{2}\lambda ^{\beta _{2}}}{%
\alpha }(1-u)^{\beta _{2}}\right] ^{-pt}.  \label{dist4}
\end{equation}%
Indeed, it is%
\begin{align*}
e^{-\frac{1}{p}\partial _{t}}&\left[ 1+\frac{a_{1}\lambda ^{\beta _{1}}}{%
\alpha }(1-u)^{\beta _{1}}+\frac{a_{2}\lambda ^{\beta _{2}}}{\alpha }%
(1-u)^{\beta _{2}}\right] ^{-pt} \\
=&\left[ 1+\frac{a_{1}\lambda ^{\beta _{1}}}{\alpha }(1-u)^{\beta _{1}}+%
\frac{a_{2}\lambda ^{\beta _{2}}}{\alpha }(1-u)^{\beta _{2}}\right] ^{-p(t-%
\frac{1}{p})} \\
=&G_{\beta _{1},\beta _{2}}(u|\alpha ,pt,\lambda )+\left[ \frac{a_{1}\lambda
^{\beta _{1}}}{\alpha }(1-u)^{\beta _{1}}+\frac{a_{2}\lambda ^{\beta _{2}}}{%
\alpha }(1-u)^{\beta _{2}}\right] G_{\beta _{1},\beta _{2}}(u|\alpha
,pt,\lambda ).
\end{align*}%
Now we only need to prove that the coefficients in the series expansion of (%
\ref{dist4}) coincide with (\ref{dist5}): for $u$ such that $\frac{%
a_{1}\lambda ^{\beta _{1}}}{\alpha }(1-u)^{\beta _{1}}+\frac{a_{2}\lambda
^{\beta _{2}}}{\alpha }(1-u)^{\beta _{2}}<1$, we can write%
\begin{eqnarray*}
&&G_{\beta _{1},\beta _{2}}(u|\alpha ,pt,\lambda ) \\
&=&\sum_{l=0}^{\infty }\binom{pt+l-1}{l}(-1)^{l}\left[ \frac{a_{1}\lambda
^{\beta _{1}}}{\alpha }(1-u)^{\beta _{1}}+\frac{a_{2}\lambda ^{\beta _{2}}}{%
\alpha }(1-u)^{\beta _{2}}\right] ^{l} \\
&=&\sum_{l=0}^{\infty }\binom{pt+l-1}{l}\left( -\frac{1}{\alpha }\right)
^{l}\sum_{r=0}^{l}\binom{l}{r}a_{1}^{r}\lambda ^{\beta _{1}r}(1-u)^{\beta
_{1}r}a_{2}^{l-r}\lambda ^{\beta _{2}(l-r)}(1-u)^{\beta _{2}(l-r)} \\
&=&\sum_{l=0}^{\infty }\binom{pt+l-1}{l}\left( -\frac{1}{\alpha }\right)
^{l}\sum_{r=0}^{l}\binom{l}{r}a_{1}^{r}\lambda ^{\beta
_{1}r}a_{2}^{l-r}\lambda ^{\beta _{2}(l-r)}\sum_{n=0}^{\infty }\binom{\beta
_{1}r+\beta _{2}(l-r)}{n}(-u)^{n}.
\end{eqnarray*}%
Thus, we get%
\begin{eqnarray}
&&\overline{\delta }_{\beta _{1},\beta _{2}}(n|\alpha ,pt,\lambda )
\label{dist6} \\
&=&\frac{(-1)^{n}}{n!}\frac{1}{\Gamma (pt)}\,\sum_{l=0}^{\infty }\Gamma
(pt+l)\left( -\frac{1}{\alpha }\right) ^{l}\sum_{r=0}^{l}\frac{%
a_{1}^{r}\lambda ^{\beta _{1}r}a_{2}^{l-r}\lambda ^{\beta _{2}(l-r)}}{%
r!(l-r)!}\frac{\Gamma (\beta _{1}r+\beta _{2}(l-r)+1)}{\Gamma (\beta
_{1}r+\beta _{2}(l-r)+1-n)}  \notag \\
&=&\frac{(-1)^{n}}{n!}\frac{1}{\Gamma (pt)}\,\sum_{r=0}^{\infty }\frac{1}{r!}%
\left( \frac{a_{1}\lambda ^{\beta _{1}}}{a_{2}\lambda ^{\beta _{2}}}\right)
^{r}\sum_{l=r}^{\infty }\frac{\Gamma (pt+l)(-1)^{l}}{\alpha ^{l}(l-r)!}\frac{%
a_{2}^{l}\lambda ^{\beta _{2}l}\Gamma (\beta _{1}r+\beta _{2}(l-r)+1)}{%
\Gamma (\beta _{1}r+\beta _{2}(l-r)+1-n)}  \notag \\
&=&\frac{(-1)^{n}}{n!}\frac{1}{\Gamma (pt)}\,\sum_{r=0}^{\infty }\frac{%
(-a_{1}\lambda ^{\beta _{1}})^{r}}{r!\alpha ^{r}}\sum_{l=0}^{\infty }\frac{%
\Gamma (pt+l+r)(-1)^{l}}{\alpha ^{l}l!}\frac{a_{2}^{l}\lambda ^{\beta
_{2}l}\Gamma (\beta _{1}r+\beta _{2}l+1)}{\Gamma (\beta _{1}r+\beta
_{2}l+1-n)},  \notag
\end{eqnarray}%
which coincides with (\ref{dist5}). The inner series converges for $%
\left\vert a_{1}\lambda ^{\beta _{1}}/\alpha \right\vert <1$, since $\Delta
=\beta _{1}-1-\beta _{1}=-1$ and $\delta =\beta _{1}^{-\beta _{1}}\beta
_{1}^{\beta _{1}}=1.$ In order to prove that (\ref{dist5}) represents a
proper distribution, we evaluate the sum of (\ref{dist6}) over all $n\in
\mathbb{Z}^{+}$:%
\begin{eqnarray*}
&&\sum_{n=0}^{\infty }\overline{\delta }_{\beta _{1},\beta _{2}}(n|\alpha
,pt,\lambda ) \\
&=&\frac{1}{\Gamma (pt)}\sum_{l=0}^{\infty }\,\frac{(-a_{2}\lambda ^{\beta
_{2}})^{l}}{\alpha ^{l}l!}\sum_{r=0}^{\infty }\frac{(-a_{1}\lambda ^{\beta
_{1}})^{r}}{r!\alpha ^{r}}\Gamma (pt+l+r)\sum_{n=0}^{\infty }\frac{(-1)^{n}}{%
n!}\frac{\Gamma (\beta _{1}r+\beta _{2}l+1)}{\Gamma (\beta _{1}r+\beta
_{2}l+1-n)} \\
&=&\frac{1}{\Gamma (pt)}\sum_{l=0}^{\infty }\,\frac{(-a_{2}\lambda ^{\beta
_{2}})^{l}}{\alpha ^{l}l!}\sum_{r=0}^{\infty }\frac{(-a_{1}\lambda ^{\beta
_{1}})^{r}}{r!\alpha ^{r}}\Gamma (pt+l+r)\sum_{n=0}^{\infty }(-1)^{n}\binom{%
\beta _{1}r+\beta _{2}l}{n} \\
&=&\frac{1}{\Gamma (pt)}\sum_{l=0}^{\infty }\,\frac{(-a_{2}\lambda ^{\beta
_{2}})^{l}}{\alpha ^{l}l!}\sum_{r=0}^{\infty }\frac{(-a_{1}\lambda ^{\beta
_{1}})^{r}}{r!\alpha ^{r}}\Gamma (pt+l+r)(1-1)^{\beta _{1}r+\beta _{2}l}=%
\frac{\Gamma (pt)}{\Gamma (pt)}=1.
\end{eqnarray*}
\end{proof}

It is easy to check that , for $a_{1}=1,$ $a_{2}=0$\ and $\beta _{1}=\beta $%
, formula (\ref{dist5}) reduces to (\ref{dis}).

\subsection{Space-fractional ``Polya-type" process}

Another generalization of the SFNB process can be obtained by considering
the following process%
\begin{equation}
\overline{W}_{\beta }(t,u)=\overline{N}_{\beta }(u,Y^{d}(t)).  \label{np}
\end{equation}%
Here $\left\{ \overline{N}_{\beta }(\cdot ,Y^{d}(t))\right\} _{t\geq 0}$ is
the space-fractional Poisson process with random parameter represented by
the process $Y^{d}(t),$ where$\ d$ is a positive constant and $\left\{
Y(t)\right\} _{t\geq 0}$ is an independent gamma subordinator with
parameters $\frac{1}{\lambda }$ and $pt$. Its distribution can be written as
follows:%
\begin{eqnarray*}
\mathbb{P}\left\{ \overline{W}_{\beta }(t,u)=n\right\} &=&\overline{\eta }%
_{\beta }(n|u,\lambda ,pt,d) \\
&=&\int_{0}^{+\infty }\mathbb{P}\left\{ \overline{N}_{\beta
}(u,y^{d})=n\right\} \frac{1}{\lambda ^{pt}\Gamma (pt)}y^{pt-1}e^{-\frac{1}{%
\lambda }y}dy \\
&=&\frac{1}{\lambda ^{pt}\Gamma (pt)}\frac{(-1)^{n}}{n!}\sum_{k=0}^{\infty }%
\frac{(-u)^{k}}{k!}\frac{\Gamma (\beta k+1)}{\Gamma (\beta k+1-n)}%
\int_{0}^{+\infty }y^{pt+\beta dk-1}e^{-\frac{1}{\lambda }y}dy \\
&=&\frac{1}{\Gamma (pt)}\frac{(-1)^{n}}{n!}\sum_{k=0}^{\infty }\frac{%
(-u\lambda ^{\beta d})^{k}}{k!}\frac{\Gamma (\beta k+1)\Gamma (pt+\beta dk)}{%
\Gamma (\beta k+1-n)} \\
&=&\frac{1}{\Gamma (pt)}\frac{(-1)^{n}}{n!}\,_{2}\Psi _{1}\left[ \left.
-u\lambda ^{\beta d}\right\vert _{(1-n,\beta )}^{(1,\beta )\;(pt,\beta d)}%
\right] .
\end{eqnarray*}%
One can check that the SFNB process can be recovered from (\ref{np}) by
setting $d=1/\beta $ and $u=1/\alpha $; in fact we have that%
\begin{eqnarray*}
\mathbb{P}\left\{ \overline{W}_{\beta }(t,1/\alpha )=n\right\} &=&\frac{%
(-1)^{n}}{n!}\frac{1}{\Gamma (pt)}\,_{2}\Psi _{1}\left[ \left. -\frac{%
\lambda }{\alpha }\right\vert _{(1-n,\beta )}^{(1,\beta )\;(pt,1)}\right] \\
&=&\overline{\eta }_{\beta }(n|\frac{1}{\alpha },\lambda ,pt,\frac{1}{\beta }%
) \\
&=&\overline{\delta }_{\beta }(n|\alpha ,pt,\lambda ^{1/\beta }).
\end{eqnarray*}%
The one-dimensional distribution of the space-fractional Polya-type process $%
\bar{\eta}_{\beta }(n|u,\lambda ,pt,d)$ satisfies the following pde:
\begin{eqnarray*}
\frac{1}{p}\partial _{t}^{\nu }\bar{\eta}_{\beta }\left( n|u,\lambda
,pt,d\right) &=&\partial _{t}^{\nu -1}\left[ \log (\alpha )-\psi (pt)\right]
\bar{\eta}_{\beta }\left( n|u,\lambda ,pt,d\right) \\
&&+\int_{0}^{\infty }(\log y)\bar{p}_{\beta }(n|u,y^{d})\partial _{t}^{\nu
-1}g(y|\frac{1}{\lambda },pt)dy,
\end{eqnarray*}%
which follows using
\begin{equation*}
\overline{\eta }_{\beta }(n|u,\lambda ,pt,d)=\int_{0}^{\infty }\overline{p}%
_{\beta }(n|u,y^{d})g(y|\frac{1}{\lambda },pt)dy
\end{equation*}%
and the fractional pde satisfied by $g(y|\frac{1}{\lambda },pt)$ given in
Lemma \ref{lem4}.

\section{The multivariate case}

We study now a multivariate version of the SFNB process introduced in
Definition 1, by considering a $d$-dimensional space-fractional Poisson
process (as defined in \cite{BEGM}) and subordinating all its components by
the same gamma subordinator.

\begin{definition}
Let $\left\{ \left( \overline{N}_{\beta }^{1}(t,\lambda _{1}),...,\overline{N%
}_{\beta }^{d}(t,\lambda _{d})\right) \right\} _{t\geq 0}$ be the $d$%
-dimensional SFPP with parameters $\lambda _{j}>0$, $j=1,...,d$, and let $%
\left\{ Y(t)\right\} _{t\geq 0}$ be a gamma subordinator, independent from $%
\overline{N}_{\beta }^{j}(t,\lambda _{j})$, for any $j.$ Then the process
defined by
\begin{equation}
\left( \overline{Q}_{\beta }^{1}(t,\lambda _{1}),...,\overline{Q}_{\beta
}^{d}(t,\lambda _{d})\right) =\left( \overline{N}_{\beta }^{1}(Y (t),\lambda
_{1}),...,\overline{N}_{\beta }^{d}(Y(t),\lambda _{d})\right) ,\qquad t\geq
0,  \label{mu}
\end{equation}%
is called multivariate SFNB process.
\end{definition}

We extend now the results of Theorem 5 to the $d$-dimensional case. We take,
for simplicity, $p=1$ so that $Y(t) \sim G(\alpha, t)$, henceforth.

\begin{theorem}
Let $\underline{n}=(n_{1},...,n_{d})$, for $n_{i}\in \mathbb{Z}^{+}$ and $%
\underline{\lambda }=(\lambda _{1},...,\lambda _{d}),$ $\lambda _{i}>0$,
such that $\sum_{j=1}^{d}\lambda _{j}<\alpha ^{1/\beta}$. Then, for $\text{ }%
t\geq 0 $, the one-dimensional distribution of the process (\ref{mu}) is
given by%
\begin{eqnarray}
\overline{\delta }_{\beta }(\underline{n}|\alpha ,t,\underline{\lambda }) &=&%
\mathbb{P}\left\{ \overline{Q}_{\beta }^{1}(t,\lambda _{1})=n_{1},...,%
\overline{Q}_{\beta }^{d}(t,\lambda _{d})=n_{d}\right\}  \notag \\
&=&\frac{(-1)^{s(\underline{n})}}{\prod\limits_{i=1}^{d}n_{i}!}\frac{1}{%
\Gamma (t)}\frac{\prod\limits_{i=1}^{d}\lambda _{i}^{n_{i}}}{\left( s(%
\underline{\lambda })\right) ^{s(\underline{n})}}\,_{2}\Psi _{1}\left[
\left. -\frac{\left( s(\underline{\lambda })\right) ^{\beta }}{\alpha }%
\right\vert _{(1-s(\underline{n}),\beta )}^{(1,\beta )\;(t,1)}\right],
\label{mu2}
\end{eqnarray}%
where $s(\underline{n})=\sum_{j=1}^{d}n_{j}$ and $s(\underline{\lambda }%
)=\sum_{j=1}^{d}\lambda _{j}$. The distribution (\ref{mu2}) satisfies the
following space-fractional equation%
\begin{equation}
(e^{-\partial _{t}}-1)\overline{\delta }_{\beta }(\underline{n}|\alpha ,t,%
\underline{\lambda })=\frac{s^{\beta }(\underline{\lambda })}{\alpha }\left(
I-\frac{1}{s(\underline{\lambda })}\sum_{j=1}^{d}\lambda _{j}B_{j}\right)
^{\beta }\overline{\delta }_{\beta }(\underline{n}|\alpha ,t,\underline{%
\lambda }),  \label{mu3}
\end{equation}%
with initial condition $\overline{\delta }_{\beta }(\underline{n}|\alpha ,0,%
\underline{\lambda })=1_{\left\{ \underline{n}=\underline{0}\right\} }$ and $%
B_{j}$ is the coordinate-by-coordinate backward operator, that is, $B_{j}f(%
\underline{n})=f(n_{1},...,n_{j}-1,...,n_{d}), 1 \leq j \leq d.$
\end{theorem}

\begin{proof}
Formula (\ref{mu2}) can be derived by applying the joint distribution of the
$d$-dimensional SF Poisson process given in Proposition 3.5 of \cite{BEGM}
(for $\nu =1$):%
\begin{eqnarray*}
\overline{\delta }_{\beta }(\underline{n}|\alpha ,t,\underline{\lambda })
&=&\int_{0}^{+\infty }\mathbb{P}\left\{ \overline{N}_{\beta }^{1}(z,\lambda
_{1})=n_{1},...,\overline{N}_{\beta }^{d}(z,\lambda _{d})=n_{d}\right\}
g(z|\alpha ,t)dz \\
&=&\frac{(-1)^{s(\underline{n})}s(\underline{n})!}{\prod%
\limits_{i=1}^{d}n_{i}!}\frac{\alpha ^{t}}{\Gamma (t)}\frac{%
\prod\limits_{i=1}^{d}\lambda _{i}^{n_{i}}}{(s(\underline{\lambda }))^{s(%
\underline{n})}}\sum_{r=0}^{\infty }\frac{(-s^{\beta }(\underline{\lambda })
)^{r}}{r!}\binom{\beta r}{s(\underline{n})}\int_{0}^{+\infty
}z^{r+t-1}e^{-\alpha z}dz \\
&=&\frac{(-1)^{s(\underline{n})}s(\underline{n})!}{\prod%
\limits_{i=1}^{d}n_{i}!}\frac{1}{\Gamma (t)}\frac{\prod\limits_{i=1}^{d}%
\lambda _{i}^{n_{i}}}{(s(\underline{\lambda }))^{s(\underline{n})}}%
\sum_{r=0}^{\infty }\frac{(-s^{\beta }(\underline{\lambda })/\alpha )^{r}}{r!%
}\binom{\beta r}{s(\underline{n})}\Gamma (r+t) \\
&=&\frac{(-1)^{s(\underline{n})}}{\prod\limits_{i=1}^{d}n_{i}!}\frac{1}{%
\Gamma (t)}\frac{\prod\limits_{i=1}^{d}\lambda _{i}^{n_{i}}}{(s(\underline{%
\lambda }))^{s(\underline{n})}}\sum_{r=0}^{\infty }\frac{(-s^{\beta }(%
\underline{\lambda })/\alpha )^{r}}{r!}\frac{\Gamma \left( \beta r+1\right)
\Gamma (r+t)}{\Gamma \left( \beta r-s(\underline{n})+1\right) },
\end{eqnarray*}%
which coincides with (\ref{mu2}). The series converges absolutely for $%
s^{\beta }(\underline{\lambda })<\alpha ,$ by Theorem 1.5, p.56 in \cite{KIL}
with $\Delta =\beta -(\beta +1)=-1,$\ $\delta =|\beta |^{-\beta }|\beta
|^{\beta }=1.$

The equation (\ref{mu3}) can be checked by means of the probability
generating function%
\begin{eqnarray}
G_{\beta }(\underline{u}|\alpha ,t,\underline{\lambda }):=
&&\sum_{n_{1},...,n_{d}\geq 0}u_{1}^{n_{1}}\cdot \cdot \cdot u_{d}^{n_{d}}%
\overline{\delta }_{\beta }(\underline{n}|\alpha ,t,\underline{\lambda })
\label{iii} \\
&=&\int_{0}^{+\infty }\sum_{n_{1},...,n_{d}\geq 0}\left(
\prod\limits_{i=1}^{d}u_{i}^{n_{i}}\right) \mathbb{P}\left\{ \overline{N}%
_{\beta }^{1}(z,\lambda _{1})=n_{1},...,\overline{N}_{\beta }^{d}(z,\lambda
_{d})=n_{d}\right\} g(z|\alpha ,t)dz  \notag \\
&=&[\text{by eq. (5) in \cite{BEGM}, for }\nu =1]  \notag \\
&=&\int_{0}^{+\infty }e^{-z\left[ \sum_{j=1}^{d}\lambda _{j}(1-u_{j})\right]
^{\beta }}g(z|\alpha ,t)dz=\left( 1+\frac{\left[ \sum_{j=1}^{d}\lambda
_{j}(1-u_{j})\right] ^{\beta }}{\alpha }\right) ^{-t}.  \notag
\end{eqnarray}%
Let now $S_{r}=\left\{ (l_{1},...,l_{d}):l_{j}\geq
0,\sum_{j=1}^{d}l_{j}=r\right\} $. We multiply equation (\ref{mu3}) by $%
\prod\limits_{i=1}^{d}u_{i}^{n_{i}}$ and sum over $n_{1},...,n_{d}$, so that
we get form the r.h.s.%
\begin{eqnarray}
&&\frac{s^{\beta }(\underline{\lambda })}{\alpha }\sum_{n_{1},...,n_{d}\geq
0}\left( \prod\limits_{i=1}^{d}u_{i}^{n_{i}}\right) \left( I-\frac{1}{s(%
\underline{\lambda })}\sum_{j=1}^{d}\lambda _{j}B_{j}\right) ^{\beta }%
\overline{\delta }_{\beta }(\underline{n}|\alpha ,t,\underline{\lambda })
\label{gp} \\
&=&\frac{s^{\beta }(\underline{\lambda })}{\alpha }\sum_{n_{1},...,n_{d}\geq
0}\left( \prod\limits_{i=1}^{d}u_{i}^{n_{i}}\right) \sum_{r=0}^{\infty }%
\binom{\beta }{r}\frac{(-1)^{r}}{(s(\underline{\lambda }))^{r}}\left(
\sum_{j=1}^{d}\lambda _{j}B_{j}\right) ^{r}\overline{\delta }_{\beta }(%
\underline{n}|\alpha ,t,\underline{\lambda })  \notag \\
&=&\frac{s^{\beta }(\underline{\lambda })}{\alpha }\sum_{n_{1},...,n_{d}\geq
0}\left( \prod\limits_{i=1}^{d}u_{i}^{n_{i}}\right) \sum_{r=0}^{\infty }%
\binom{\beta }{r}\frac{(-1)^{r}}{(s(\underline{\lambda }))^{r}}%
\sum_{l_{1},...,l_{d}\in S_{r}}\binom{r}{l_{1},...,l_{d}}\prod%
\limits_{i=1}^{d}\lambda _{i}^{l_{i}}B_{i}^{l_{i}}\overline{\delta }_{\beta
}(\underline{n}|\alpha ,t,\underline{\lambda })  \notag \\
&=&\frac{s^{\beta }(\underline{\lambda })}{\alpha }\sum_{r=0}^{\infty }%
\binom{\beta }{r}\frac{(-1)^{r}}{(s(\underline{\lambda }))^{r}}%
\sum_{l_{1},...,l_{d}\in S_{r}}\binom{r}{l_{1},...,l_{d}}%
\sum_{n_{1},...,n_{d}\geq 0}\prod\limits_{i=1}^{d}u_{i}^{n_{i}}\lambda
_{i}^{l_{i}}B_{i}^{l_{i}}\overline{\delta }_{\beta }(\underline{n}|\alpha ,t,%
\underline{\lambda })  \notag \\
&=&\frac{s^{\beta }(\underline{\lambda })}{\alpha }\sum_{r=0}^{\infty }%
\binom{\beta }{r}\frac{(-1)^{r}}{(s(\underline{\lambda }))^{r}}%
\sum_{l_{1},...,l_{d}\in S_{r}}\binom{r}{l_{1},...,l_{d}}\left(
\prod\limits_{i=1}^{d}u_{i}^{l_{i}}\lambda _{i}^{l_{i}}\right)
\sum_{n_{1},...,n_{d}\geq 0}\prod\limits_{i=1}^{d}u_{i}^{n_{i}}\overline{%
\delta }_{\beta }(\underline{n}|\alpha ,t,\underline{\lambda })  \notag \\
&=&\frac{s^{\beta }(\underline{\lambda })}{\alpha }\left( I-\frac{1}{s(%
\underline{\lambda })}\sum_{j=1}^{d}\lambda _{j}u_{j}\right) ^{\beta
}G_{\beta }(\underline{u}|\alpha ,t,\underline{\lambda }).  \notag
\end{eqnarray}%
From the l.h.s. of (\ref{mu3}) we have instead that%
\begin{eqnarray*}
&&(e^{-\partial _{t}}-1)G_{\beta }(\underline{u}|\alpha ,t,\underline{%
\lambda }) \\
&=&G_{\beta }(\underline{u}|\alpha ,t-1,\underline{\lambda })-G_{\beta }(%
\underline{u}|\alpha ,t,\underline{\lambda }) \\
&=&\left( 1+\frac{\left[ \sum_{j=1}^{d}\lambda _{j}(1-u_{j})\right] ^{\beta }%
}{\alpha }\right) ^{-(t-1)}-\left( 1+\frac{\left[ \sum_{j=1}^{d}\lambda
_{j}(1-u_{j})\right] ^{\beta }}{\alpha }\right) ^{-t} \\
&=&G_{\beta }(\underline{u}|\alpha ,t,\underline{\lambda })\left[ 1+\frac{%
\left[ \sum_{j=1}^{d}\lambda _{j}(1-u_{j})\right] ^{\beta }}{\alpha }-1%
\right] ,
\end{eqnarray*}%
which coincides with (\ref{gp}).
\end{proof}

For $\beta =1$, we obtain the joint distribution of the multivariate NB
process, which is therefore defined as%
\begin{equation}
\left( N^{1}(Y(t),\lambda _{1}),...,N^{d}(Y(t),\lambda _{d})\right) ,\qquad
t\geq 0,  \label{ggg}
\end{equation}%
where, in this case, the $d$ components of the leading vector $\left(
N^{1}(t,\lambda _{1}),...,N^{d}(t,\lambda _{d})\right) $ are mutually
independent, while the common subordination by the gamma process introduces
correlations among the components of the process.

\begin{corollary}
The joint distribution of the multivariate NB process defined in (\ref{ggg})
is given by%
\begin{eqnarray*}
\overline{\delta }_{1}(\underline{n}|\alpha ,t,\underline{\lambda }) &=&%
\mathbb{P}\left\{ N^{1}(Y(t),\lambda _{1})=n_{1},...,N^{d}(Y(t),\lambda
_{d})=n_{d}\right\} \\
&=&\frac{\alpha ^{t}s(\underline{n})!}{\left( \alpha +s(\underline{\lambda }%
)\right) ^{s(\underline{n})+t}}\binom{s(\underline{n})+t-1}{t-1}%
\prod\limits_{i=1}^{d}\frac{\lambda _{i}^{n_{i}}}{n_{i}!}
\end{eqnarray*}%
and satisfies the p.d.e.%
\begin{equation}
(e^{-\partial _{t}}-1)\overline{\delta }_{1}(n_{1},...,n_{d}|\alpha ,t,%
\underline{\lambda })=\frac{s(\underline{\lambda })}{\alpha }\overline{%
\delta }_{1}(\underline{n}|\alpha ,t,\underline{\lambda })-\frac{1}{\alpha }%
\sum_{j=1}^{d}\lambda _{j}\overline{\delta }_{1}(n_{1},..,n_{i-1},.,n_{d}|%
\alpha ,t,\underline{\lambda }),  \label{mu4}
\end{equation}%
with initial condition $\overline{\delta }_{1}(\underline{n}|\alpha ,0,%
\underline{\lambda })=1_{\left\{ \underline{n}=\underline{0}\right\} }.$
\end{corollary}

\begin{proof}
From (\ref{mu2}) we get%
\begin{eqnarray*}
&&\mathbb{P}\left\{ N^{1}(Y(t),\lambda _{1})=n_{1},...,N^{d}(Y(t),\lambda
_{d})=n_{d}\right\} \\
&=&\frac{(-1)^{s(\underline{n})}}{\prod\limits_{i=1}^{d}n_{i}!}\frac{1}{%
\Gamma (t)}\frac{\prod\limits_{i=1}^{d}\lambda _{i}^{n_{i}}}{(s(\underline{%
\lambda }))^{s(\underline{n})}}\sum_{r=s(\underline{n})}^{\infty }\frac{(-s(%
\underline{\lambda })/\alpha )^{r}}{(r-s(\underline{n}))!}\Gamma (r+t) \\
&=&\frac{1}{\Gamma (t)\alpha ^{s(\underline{n})}}\prod\limits_{i=1}^{d}\frac{%
\lambda _{i}^{n_{i}}}{n_{i}!}\sum_{l=0}^{\infty }\frac{(-s(\underline{%
\lambda })/\alpha )^{l}}{l!}\Gamma (l+s(\underline{n})+t) \\
&=&\frac{\Gamma (s(\underline{n})+t)}{\Gamma (t)\alpha ^{s(\underline{n})}}%
\prod\limits_{i=1}^{d}\frac{\lambda _{i}^{n_{i}}}{n_{i}!}\sum_{l=0}^{\infty }%
\binom{l+s(\underline{n})+t-1}{l}(-s(\underline{\lambda })/\alpha )^{l} \\
&=&\binom{s(\underline{n})+t-1}{t-1}\frac{s(\underline{n})!}{\alpha ^{s(%
\underline{n})}}\prod\limits_{i=1}^{d}\frac{\lambda _{i}^{n_{i}}}{n_{i}!}%
\left( 1+\frac{s(\underline{\lambda })}{\alpha }\right) ^{-s(\underline{n}%
)-t},
\end{eqnarray*}%
where, in the last step, we have used again the identity (\ref{id}).
Equation (\ref{mu4}) follows immediately from (\ref{mu3}) for $\beta =1.$%
\newline
\end{proof}

\begin{remark}
In the non-fractional case, we can evaluate the covariance between two
components of the process (\ref{ggg}), by deriving the moments generating
function; we assume for simplicity $\lambda _{j}=\lambda $, for any $%
j=1,...,d,$ and perform calculations similar to those in (\ref{iii}):
\begin{eqnarray}
M(u_{j},u_{l};t)&:=&\mathbb{E}e^{\{u_{j}N^{j}(Y (t))+u_{l}N^{l}(Y (t))\}}
\label{mm} \\
&=&\frac{\alpha ^{t}}{\left[ \alpha +2\lambda -\lambda (e^{u_{j}}+e^{u_{l}})%
\right] ^{t}}.  \notag
\end{eqnarray}%
Thus, for any $j\neq l,$%
\begin{equation*}
\mathbb{C}ov\left[ N^{j}(Y (t)),N^{l}(Y (t))\right] =\frac{\lambda ^{2}t}{%
\alpha ^{2}},
\end{equation*}%
which is obviously positive and linearly increasing with $t.$
\end{remark}

\

On the other hand, in the fractional case, the process has no finite
moments. It is a multivariate L\'{e}vy process and its Laplace exponent is
given by
\begin{equation*}
\psi (u)=\ln \left( 1+\frac{\left[ \sum_{j=1}^{d}\lambda _{j}(1-e^{-u_{j}})%
\right] ^{\beta }}{\alpha }\right) .
\end{equation*}%
Its discrete L\'{e}vy measure can be derived, by applying again Theorem
30.1, p. 197 in \cite{SAT} as follows: let us write $\underline{n}\succ
\underline{0}$ to mean that $n_{j}\geq 0$ for any $j=1,...,d$, but $%
\underline{n}\neq \underline{0}$, then%
\begin{eqnarray*}
\nu _{\beta }(\cdot ) &=&\int_{0}^{+\infty }\sum_{\underline{n}\succ
\underline{0}}\mathbb{P}\left\{ N_{\beta }^{1}(t,\lambda
_{1})=n_{1},...,N_{\beta }^{d}(t,\lambda _{d})=n_{d}\right\} \delta
_{\left\{ \underline{n}\right\} }(\cdot )\mu _{\Gamma }(s)ds \\
&=&\sum_{\underline{n}\succ \underline{0}}\frac{(-1)^{s(\underline{n})}s(%
\underline{n})!}{\prod\limits_{i=1}^{d}n_{i}!}\frac{\prod\limits_{i=1}^{d}%
\lambda _{i}^{n_{i}}}{(s(\underline{\lambda }))^{s(\underline{n})}}%
\sum_{r=0}^{\infty }\frac{(-s^{\beta }(\underline{\lambda })/\alpha )^{r}}{r!%
}\binom{\beta r}{s(\underline{n})}\delta _{\left\{ \underline{n}\right\}
}(\cdot )\int_{0}^{+\infty }s^{r-1}e^{-\alpha s}ds \\
&=&\sum_{\underline{n}\succ \underline{0}}\frac{(-1)^{s(\underline{n})}s(%
\underline{n})!}{\prod\limits_{i=1}^{d}n_{i}!}\frac{\prod\limits_{i=1}^{d}%
\lambda _{i}^{n_{i}}}{(s(\underline{\lambda }))^{s(\underline{n})}}\delta
_{\left\{ \underline{n}\right\} }(\cdot )\sum_{r=1}^{\infty }\frac{%
(-s^{\beta }(\underline{\lambda })/\alpha )^{r}}{r}\binom{\beta r}{s(%
\underline{n})} \\
&=&\sum_{\underline{n}\succ \underline{0}}\frac{(-1)^{s(\underline{n})}}{%
\prod\limits_{i=1}^{d}n_{i}!}\frac{\prod\limits_{i=1}^{d}\lambda _{i}^{n_{i}}%
}{(s(\underline{\lambda }))^{s(\underline{n})}}\delta _{\left\{ \underline{n}%
\right\} }(\cdot )\,_{2}\Psi _{1}\left[ \left. -\frac{\left( s(\underline{%
\lambda })\right) ^{\beta }}{\alpha }\right\vert _{(1-s(\underline{n}),\beta
)}^{(1,\beta )\;(0,1)}\right] .
\end{eqnarray*}%
The series converges absolutely for $s^{\beta }(\underline{\lambda })<\alpha
$. In the special case $\beta =1,$ we get the L\'{e}vy measure of the
multivariate NB process:%
\begin{eqnarray*}
\nu _{1}(\cdot ) &=&\sum_{n_{1},...,n_{d}\geq 1}^{\infty }\frac{(-1)^{s(%
\underline{n})}s(\underline{n})!}{\prod\limits_{i=1}^{d}n_{i}!}\frac{%
\prod\limits_{i=1}^{d}\lambda _{i}^{n_{i}}}{(s(\underline{\lambda }))^{s(%
\underline{n})}}\delta _{\left\{ \underline{n}\right\} }(\cdot )\sum_{r=s(%
\underline{n})}^{\infty }\frac{(-s(\underline{\lambda })/\alpha )^{r}}{r}%
\binom{r}{s(\underline{n})} \\
&=&\sum_{n_{1},...,n_{d}\geq 1}^{\infty }\frac{s(\underline{n})!}{%
\prod\limits_{i=1}^{d}n_{i}!}\frac{\prod\limits_{i=1}^{d}\lambda _{i}^{n_{i}}%
}{\alpha ^{s(\underline{n})}}\delta _{\left\{ \underline{n}\right\} }(\cdot
)\sum_{l=0}^{\infty }\frac{(-s(\underline{\lambda })/\alpha )^{l}}{l+s(%
\underline{n})}\binom{l+s(\underline{n})}{s(\underline{n})}
\end{eqnarray*}
\begin{eqnarray*}
&=&\sum_{n_{1},...,n_{d}\geq 1}^{\infty }\frac{(s(\underline{n})-1)!}{%
\prod\limits_{i=1}^{d}n_{i}!}\frac{\prod\limits_{i=1}^{d}\lambda _{i}^{n_{i}}%
}{\alpha ^{s(\underline{n})}}\delta _{\left\{ \underline{n}\right\} }(\cdot
)\sum_{l=0}^{\infty }(-s(\underline{\lambda })/\alpha )^{l}\binom{l+s(%
\underline{n})-1}{l} \\
&=&\sum_{n_{1},...,n_{d}\geq 1}^{\infty }\frac{(s(\underline{n})-1)!}{%
\prod\limits_{i=1}^{d}n_{i}!}\frac{\prod\limits_{i=1}^{d}\lambda _{i}^{n_{i}}%
}{\left[ \alpha +s(\underline{\lambda })\right] ^{s(\underline{n})}}\delta
_{\left\{ \underline{n}\right\} }(\cdot ),
\end{eqnarray*}%
which, for $d=1$, coincides with the well-known L\'{e}vy measure of the
standard NB process.

\section{A simulation of the SFNB process}
In this section, we discuss an algorithm for the simulation of Poisson process, gamma process,
$\beta$-stable subordinator and finally for the SFNB process.

\subsection{Simulation of the Poisson process}\label{poisson-sec}

\noindent The following algorithm for the simulation of the Poisson process with rate $\lambda$ is used.\\\\
{\bf  Algorithm 1: Simulation of the Poisson process}
\begin{enumerate}[(a)]
    \item Generate $n$ independent $U_i\sim U(0,1)$, $i = 1, 2, \ldots,n$.
    \item  Set $Y_i=-\frac{\log(1-U_i)}{\lambda}$ so that $Y_i\sim $ Exp($\lambda$).
    \item  Set $X_i=\sum_{j=1}^{i}Y_j$.
    \end{enumerate}
Then $X_i$ denotes the time at which $N(t)$ increases by 1 for the
$i$-th time.

\subsection{Simulation of the gamma process}\label{gamma-sec}

The sample paths of the $G(\alpha,pt)$ process in $[0,T]$ is generated using the ``gamma sequential sampling (GSS)'' (see \cite[p. 321]{AVR})
algorithm which is given below:\\
\blockquote{$G(0)=0;$\\
    $h=2^{-k}T;$\\
    For $i=1$ to $2^k$\\
    \hspace*{0.7cm} Generate $Q\sim G(\alpha,ph)\stackrel{d}{=}  G(\mu/\nu,\mu^2h/\nu);$\\
    \hspace*{0.7cm} $G(ih)=G((i-1)h)+Q$;\\
    Next $i~~~~~~~~~~~~$\\}
{\bf Algorithm 2:   Discretized trajectory for gamma process}\\\\
{\it Fix the parameters $\alpha$ and $p$ for gamma process.}
\begin{enumerate}[(a)]
    \item Choose $n$ uniformly spaced time points $t_1,t_2,\ldots,t_n.$ with $h=t_1=t_2-t_1.$
    \item Generate or simulate $n$ independent gamma random variables
    $Q_i\sim G(\alpha,ph).$
    \item The discretized sample path of $Y(t)$ at $t_i$ is $Y(ih)=Y(t_i)=\sum_{j=1}^{i}Q_{j}$ with $Q_0=0.$

\end{enumerate}
\subsection{Simulation of $\beta$-stable subordinator $D_\beta(t),~0<\beta<1$ with support $(0,\infty)$}\label{stable-sec}
\noindent Let $D_\beta(t),~0<\beta<1$, be a $\beta$-stable
subordinator with Laplace transform
\begin{equation}\label{stable-LT}
\mathbb{E}[e^{-sD_\beta(t)}]=e^{-ts^\beta}.
\end{equation}
Since $D_\beta(t)$ is self-similar with index $1/\beta$, we have
\begin{equation}\label{stable-SS2}
D_\beta(t)\stackrel{d}{=}t^{1/\beta}D_\beta(1).
\end{equation}
Let $U$ and $W$ be independent random variables, where $U\sim
U[0,\pi]$, and $W\sim$ Exp(1). It is shown in Corollary 4.1 of
\cite{KAN} that
\begin{equation}\label{stable-kanter}
D_\beta(1)=\left(\frac{a(U)}{W}\right)^{(1-\beta)/\beta},
\end{equation}
where
\begin{equation}\label{stable-kanter2}
a(u)=\left(\frac{\sin\beta u}{\sin
u}\right)^{1/(1-\beta)}\left(\frac{\sin(1-\beta)u}{\sin\beta
u}\right).
\end{equation}
Let $t_1<t_2<\cdots<t_n$ be the time points. Then the increments
\begin{equation}\label{stable-inc}
D_\beta(t_{i})-D_\beta(t_{i-1})\stackrel{d}{=}
(t_i-t_{i-1})^{1/\beta}D_\beta(1).
\end{equation}
From \eqref{stable-kanter} and \eqref{stable-kanter2}, we get
\begin{align}
D_\beta(t_{i})-D_\beta(t_{i-1})&\stackrel{d}{=}
(t_i-t_{i-1})^{1/\beta}\frac{(\sin\beta
U)(\sin(1-\beta)U)^{(1-\beta)/\beta}}{(\sin
U)^{1/\beta}W^{(1-\beta)/\beta}}.
\end{align}\\
\noindent{\bf Algorithm 3:   Discretized trajectory for
$\beta$-stable
    subordinator}\\
\begin{enumerate}[(a)]
    \item Choose $n$ time points $t_1,t_2,\ldots,t_n.$
    \item Simulate $U_1,\ldots,U_n$, where $U_i\sim U[0,\pi]$, and $W_1,\ldots,W_n$, where $W_i\sim $Exp(1).
    \item Compute the increments
    \begin{equation*}
\Delta
D_\beta^{(i)}=D_\beta(t_i)-D_\beta(t_{i-1})=(t_i-t_{i-1})^{1/\beta}\frac{(\sin\beta
U_i)(\sin(1-\beta)U_i)^{(1-\beta)/\beta}}{(\sin
U_i)^{1/\beta}W_i^{(1-\beta)/\beta}},
    \end{equation*}
    for $1\leq i\leq n$ with $D_\beta(0)=0.$
    \item The discretized sample path of $D_\beta(t)$ at $t_i$ is $D_\beta(t_i)=\sum_{j=1}^{i}\Delta D_\beta^{(j)}.$ \\
    \end{enumerate}

\noindent The simulation of the SFNB process is done using the following algorithm.\\

\noindent {\bf Algorithm 4:  Simulation of the SFNB process}\\\\
{\it Choose first the parameters $\alpha$ and $p$ of the gamma subordinator, the rate parameter $\lambda>0$ for the Poisson process and $\beta,~0<\beta<1$ for the $\beta$-stable subordinator}.\\\\
{\it Step 1:} Fix the time $T$ for the time interval $[0,T]$. Choose time points $0<t_1<\ldots<t_n<T,$ which are uniformly spaced with $h=t_2-t_1$.\\\\
{\it Step 2:} Simulate the values $Y(t_i)$ of the gamma process for $0<t_1<\ldots<t_n<T,$ using Algorithm 2.\\\\
{\it Step 3:} Using the values $Y(t_i)$ generated in {\it Step 2}, as time points, simulate the trajectory of the $\beta$-stable subordinator $D_\beta(Y(t_i))$ using Algorithm 3.\\\\
{\it Step 4:} Using $D_\beta(Y(t_i))$ obtained in {\it Step 3}, as the time points, compute $N(D_\beta(Y(t_i)))$ of the Poisson process using Algorithm 1.\\\\
The above algorithm is implemented in Matlab 2015b to obtain the
sample paths (Figures 3 and 4) of the SFNB process for certain
chosen values of the parameter.
\begin{figure}[!hb]
    \caption{Sample paths for the SFNB process for $\beta=0.95,\alpha=1.2,p=4$ and $\lambda=2.$}
    \includegraphics[width=0.8\textwidth]{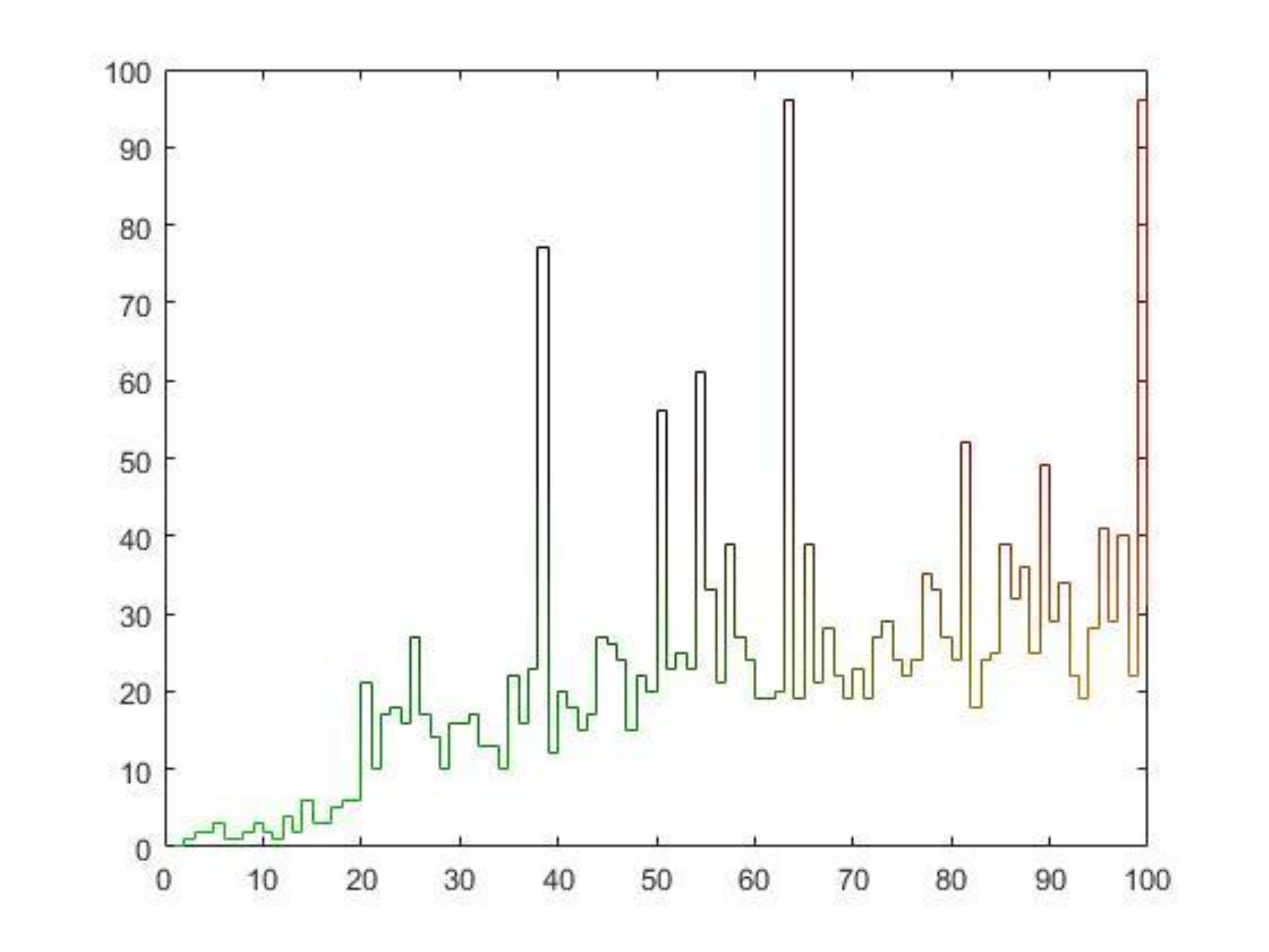}
\end{figure}
\clearpage
\begin{figure}[ht]
    \caption{Sample paths for the SFNB process for $\beta=0.8,\alpha=2,p=1.5$ and $\lambda=0.6.$}
    \centering
    \includegraphics[width=0.8\textwidth]{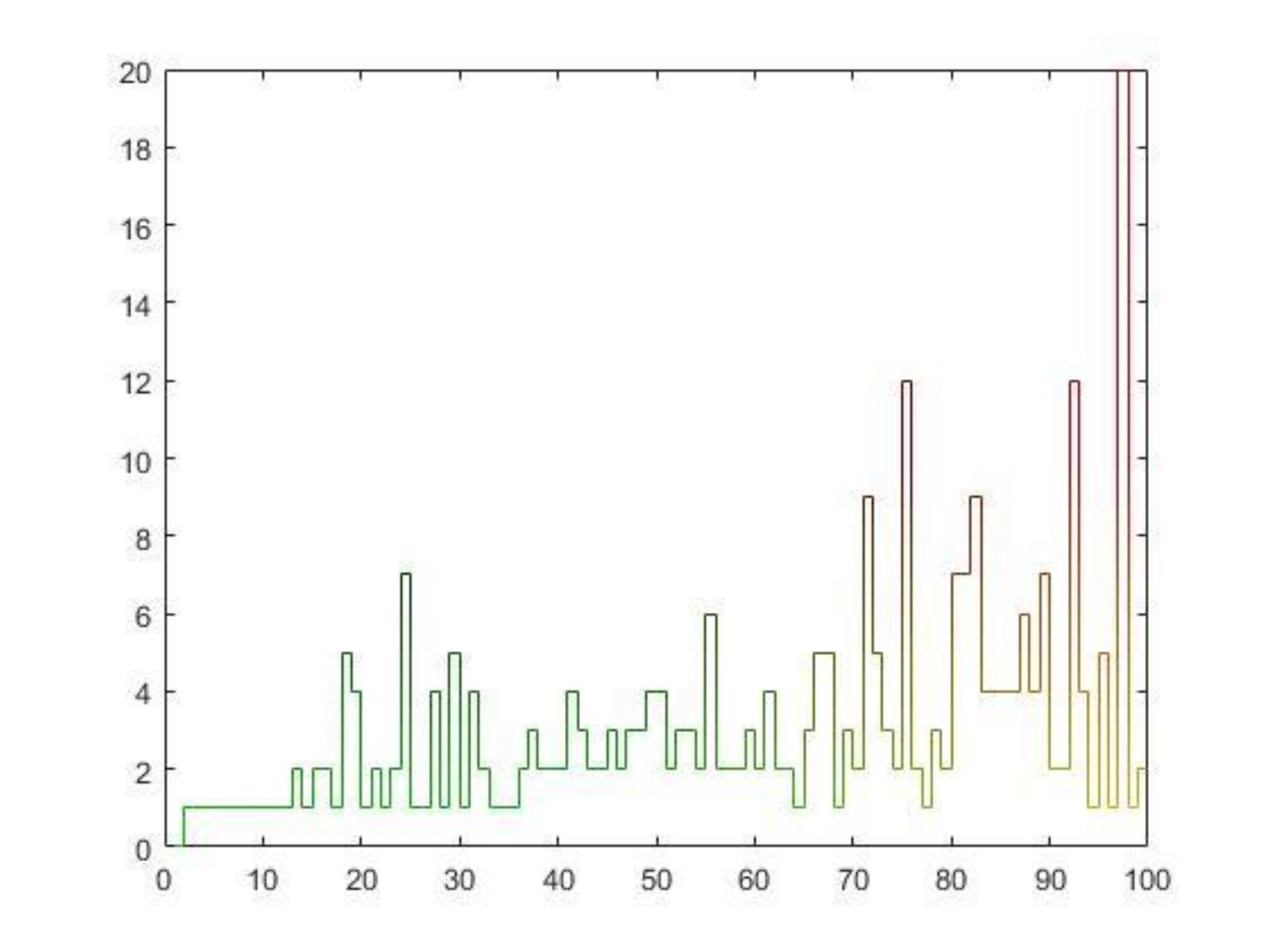}
\end{figure}

\section*{Acknowledgement}

The authors thank Mr. A. Maheshwari for his simulations and
computational help. They are also grateful to Prof. C. Macci for
his careful reading of the paper and suggesting some useful comments.
%\printbibliography

\end{document}